\documentclass[a4paper,12pt,final]{amsart}
\usepackage{times,a4wide,mathrsfs,amssymb,amsmath,amsthm,enumerate,xypic,tikzsymbols,dsfont}

\newcommand{\C}{\mathbb{C}}

\newcommand{\QQ}{\mathbb{Q}}
\newcommand{\NN}{\mathbb{N}}
\newcommand{\PP}{\mathbb{P}}

\newcommand{\OO}{\mathcal O}

\newcommand{\YY}{\mathcal Y}

\newcommand{\WW}{\mathcal W}
\newcommand{\TT}{\mathcal T}

\newcommand{\codim}{\hbox{codim}}

\newcommand{\rom}{\romannumeral}

\DeclareMathOperator{\rank}{rank}

\DeclareMathOperator{\Gr}{Gr}

\newtheorem{theorem}{Theorem}[section]

\newtheorem{corollary}[theorem]{Corollary}
\newtheorem{proposition}[theorem]{Proposition}
\newtheorem{conjecture}[theorem]{Conjecture}
\newtheorem{convention}{Conventions}

\newtheorem{nonumbering}{Theorem}

\newtheorem{nonumberingc}{Corollary}

\theoremstyle{definition}
\newtheorem{remark}[theorem]{Remark}

\newtheorem{nonumberingt}{Acknowledgments}

\begin{document}

\author[Robert Laterveer]
{Robert Laterveer}

\address{Institut de Recherche Math\'ematique Avanc\'ee,
CNRS -- Universit\'e 
de Strasbourg,\
7 Rue Ren\'e Des\-car\-tes, 67084 Strasbourg CEDEX,
FRANCE.}
\email{robert.laterveer@math.unistra.fr}

\title{On the Chow groups of Pl\"ucker hypersurfaces in Grassmannians}

\begin{abstract} Motivated by the generalized Bloch conjecture, we formulate a conjecture about the Chow groups of Pl\"ucker hypersurfaces
in Grassmannians. We prove weak versions of this conjecture.
 \end{abstract}


\maketitle

\section{Introduction}

\noindent
Given a smooth projective variety $Y$ over $\C$, let $A_i(Y):=CH_i(Y)_{\QQ}$ denote the Chow groups of $Y$ (i.e. the groups of $i$-dimensional algebraic cycles on $Y$ with $\QQ$-coefficients, modulo rational equivalence). Let $A_i^{hom}(Y)\subset A_i(Y)$ denote the subgroup of homologically trivial cycles.

The ``generalized Bloch conjecture'' \cite[Conjecture 1.10]{Vo} predicts that the Hodge level of the cohomology of $Y$ should have an influence on the size of the Chow groups of $Y$. In case $Y$ is a surface, this is the notorious Bloch conjecture, which is still an open problem. In the case of hypersurfaces in projective space, the precise prediction is as follows:

\begin{conjecture}\label{conj0} Let $Y\subset\PP^{n}$ be a smooth hypersurface of degree $d$. Then
  \[ A_i^{hom}(Y)=0\ \ \ \forall\ i\le {n\over d} -1\ .\]
\end{conjecture}

Conjecture \ref{conj0} is still open; partial results have been obtained in \cite{Lew}, \cite{V96}, \cite{Ot}, \cite{ELV}, \cite{HI}.

In the case of Pl\"ucker hypersurfaces in Grassmannians, we hazard the following prediction (cf. subsection \ref{ss:mot} below for motivation):

\begin{conjecture}\label{conj} Let $\Gr(k,n)$ denote the Grassmannian of $k$-dimensional subspaces of an $n$-dimensional vector space, and let
  \[ Y= \Gr(k,n)\cap H \ \ \ \subset\ \PP^{{n\choose k}-1} \]
  be a smooth hyperplane section (with respect to the Pl\"ucker embedding). Then
    \[ A_i^{hom}(Y)=0\ \ \ \forall\  i\le n-2\ .\]
    \end{conjecture}

    The main result of this note is that a weak version of Conjecture \ref{conj} is true:
    
    \begin{nonumbering}[=Theorem \ref{main}] Let 
  \[ Y= \Gr(k,n)\cap H \ \ \ \subset\ \PP^{{n\choose k}-1} \]
  be a smooth hyperplane section (with respect to the Pl\"ucker embedding). Then
    \[ A_i^{hom}(Y)=0\ \ \ \forall i\le n-k\ .\]
    \end{nonumbering}

    The argument proving Theorem \ref{main} is very easy and straightforward; it combines the recent construction of {\em jumps\/} among subvarieties of Grassmannians \cite{BFM} and
    a motivic version of the Cayley trick \cite{Ji}.

    In some cases, we can do better and the conjecture is completely satisfied:
    
   \begin{nonumbering}[=Theorem \ref{main2}] Let 
  \[ Y= \Gr(3,n)\cap H \ \ \ \subset\ \PP^{{n\choose 3}-1} \]
  be a smooth hyperplane section (with respect to the Pl\"ucker embedding). 
  Assume $n\le 13$, $n\not=12$.
  Then
    \[ A_i^{hom}(Y)=0\ \ \ \forall i\le n-2\ .\]
    \end{nonumbering}  
    
 The $n=10$ case of Theorem \ref{main2} was already proven by Voisin as an application of her technique of {\em spread\/} of algebraic cycles \cite[Theorem 2.4]{V1}. Hyperplane sections $Y\subset\Gr(3,10)$ are also known as {\em Debarre--Voisin hypersurfaces\/}, because they give rise to the Debarre--Voisin hyperk\"ahler fourfolds \cite{DV}. The new proof of \cite[Theorem 2.4]{V1} provided by Theorem \ref{main2} does not rely on Voisin's spread technique, nor on the relation with hyperk\"ahler fourfolds. 
    
As a consequence of Theorem \ref{main2}, some instances of the generalized Hodge conjecture are verified:

\begin{nonumberingc}[=Corollary \ref{ghc}] Let $Y$ be as in Theorem \ref{main2}. Then $H^{\dim Y}(Y,\QQ)$ is supported on a subvariety of codimension $n-1$.
\end{nonumberingc}

As another consequence, we find some new examples of varieties with finite-dimensional motive:

\begin{nonumberingc}[=Corollary \ref{cor2}] Let  
\[ Y= \Gr(3,9)\cap H \ \ \ \subset\ \PP^{83} \]
  be a smooth hyperplane section (with respect to the Pl\"ucker embedding). 
  Then $Y$ has finite-dimensional motive (in the sense of \cite{Kim}).
  \end{nonumberingc}
  
 Varieties $Y$ as in Corollary \ref{cor2} are studied in \cite[Section 5.1]{BFM}, where they are related to Coble cubics and abelian surfaces.

 \vskip0.5cm

\begin{convention} In this note, the word {\sl variety\/} will refer to a reduced irreducible scheme of finite type over $\C$. A {\sl subvariety\/} is a (possibly reducible) reduced subscheme which is equidimensional. 

{\bf All Chow groups will be with rational coefficients}: we denote by $A_j(Y):=CH_j(Y)_{\QQ} $ the Chow group of $j$-dimensional cycles on $Y$ with $\QQ$-coefficients; for $Y$ smooth of dimension $n$ the notations $A_j(Y)$ and $A^{n-j}(Y)$ are used interchangeably. 
The notations $A^j_{hom}(Y)$ and $A^j_{AJ}(X)$ will be used to indicate the subgroup of homologically trivial (resp. Abel--Jacobi trivial) cycles.
%

For a vector bundle $E$, we write $\PP(E)$ for $\hbox{Proj}(\oplus_{m>0} \hbox{Sym}^m E)$.
\end{convention}

 \vskip0.5cm

\section{Preliminaries}

 \subsection{Motivating the conjecture}
 \label{ss:mot}
 
 \begin{theorem}[Bernardara--Fatighenti--Manivel \cite{BFM}, Kuznetsov \cite{Ku}]\label{coho} Let 
   \[ Y= \Gr(k,n)\cap H \ \ \ \subset\ \PP^{{n\choose k}-1} \]
  be a smooth hyperplane section (with respect to the Pl\"ucker embedding). 
  Assume either $n>3k>6$, or $n$ and $k$ are coprime. Then $Y$ has Hodge coniveau $n-1$. More precisely, the Hodge numbers verify
    \[   h^{p,\dim Y-p}(Y) =\begin{cases}   1 & \hbox{if}\ p=n-1\ ,\\
                                                                 0 & \hbox{if}\ p<n-1\ .\\
                                                                 \end{cases}\]
                   \end{theorem}
                   
                   \begin{proof} The case $n>3k>6$ is \cite[Theorem 3]{BFM}. In case $n$ and $k$ are coprime, Kuznetsov \cite[Corollary 4.4]{Ku} has constructed an exceptional collection for the derived category of $Y$ whose right orthogonal is a Calabi--Yau category of dimension $k(n-k)+1-2n$. Taking Hochschild homology, one obtains the assertion about the Hodge numbers.
                   \end{proof}
  
  As mentioned in \cite{BFM}, the assumptions on $n$ and $k$ are probably not optimal. (And in view of the examples given in loc. cit., it seems likely that for {\em any\/} $n,k$, the Hodge coniveau of $Y$ is $\ge n-1$, while one needs some condition on $n,k$ to get an equality.)

    The {\em generalized Bloch conjecture\/} \cite[Conjecture 1.10]{Vo} predicts that any variety $Y$ with Hodge coniveau $\ge c$ has
    \[ A_i^{hom}(Y)=0\ \ \ \forall\ i<c\ .\]                                                             
     This motivates Conjecture \ref{conj}. Note that at least for $n>3k>6$ (or $n$ and $k$ coprime), the bound of Conjecture \ref{conj} is optimal: assuming $A_i^{hom}(Y)=0$ for $j\le n-1$ and applying the Bloch--Srinivas argument \cite{BS}, one would get the vanishing $h^{n-1,k(n-k)-n}(Y)=0$, contradicting Theorem \ref{coho}.

 \subsection{Jumps}
 
 \begin{proposition}[Bernardara--Fatighenti--Manivel \cite{BFM}]\label{jump} Let 
   \[ Y= \Gr(k,n)\cap H \ \ \ \subset\ \PP^{{n\choose k}-1} \]
  be a general hyperplane section (with respect to the Pl\"ucker embedding). 
  There is a Cartesian diagram
     \[ \begin{array}[c]{ccccc}
           F & \hookrightarrow& q^\ast Y & \xrightarrow{q} &    Y\\
           &&&&\\
           \downarrow{} &&\ \ \downarrow{\scriptstyle p}&&\\
           &&&&\\
           T & \hookrightarrow & \Gr(k-1,n)&&\\
           \end{array} \]
           Here the morphism $q$ is a  $\PP^{k-1}$-bundle, and the morphism
            $p$ is a $\PP^{n-k-1}$-bundle over $\Gr(k-1,n)\setminus T$ and a $\PP^{n-k}$-bundle over $T$.
           The subvariety $T\subset \Gr(k-1,n)$ is smooth of codimension $n-k+1$, given by a section of $Q^\ast(1)$ (where $Q$ denotes the universal quotient bundle on $\Gr(k-1,n)$).
           \end{proposition}
           
           \begin{proof} This is a special case of the construction of a {\em jump\/} in \cite[Section 3.3]{BFM}. The idea is to consider the flag variety $Fl(k-1,k,n)$ as a correspondence
           \[   \begin{array}[c]{ccc}
                    Fl(k-1,k,n)&\xrightarrow{q}& \Gr(k,n)\\
                    &&\\
                   \ \  \downarrow{\scriptstyle p}&&\\
                    &&\\
                    \Gr(k-1,n)\\
                    \end{array}                      \]
        and look at what happens over $Y$. The hyperplane $Y\subset\Gr(k,n)$ corresponds to a $k$-form $\Omega$ on an $n$-dimensional vector space. The variety $q^\ast Y\subset Fl(k-1,k,n)$ is defined by $q^\ast\Omega$; this is the inverse image of $Y$ under $q$. The flag variety $Fl(k-1,k,n)$ can be identified with
 the bundle of hyperplanes $\PP(Q^\ast(1))$ on $\Gr(k-1,n)$, and the locus $T\subset \Gr(k-1,n)$ where the fiber dimension of $p\colon q^\ast Y\to \Gr(k-1,n)$ jumps is the zero locus of the section
 of $Q^\ast(1)$ defined by $\Omega$. 
For general $Y$, the locus $T$ will be smooth of codimension equal to $\rank Q^\ast(1)=n-k+1$. (For the smoothness of $T$, we note that $Q^\ast(1)$ is globally generated and so the universal family $\TT$ of zero loci of sections of $Q^\ast(1)$ is a projective bundle over $Gr(k,n)$ hence $\TT$ is smooth; the smoothness of general $T$ then follows from generic smoothness applied to $\TT\to \PP H^0(\Gr(k,n),Q^\ast(1))$.)                   
                     \end{proof}

 \subsection{Cayley's trick and Chow groups}

\begin{theorem}[Jiang \cite{Ji}]\label{ji} Let $ E\to X$ be a vector bundle of rank $r\ge 2$ over a smooth projective variety $X$, and let $T:=s^{-1}(0)\subset X$ be the zero locus of a section $s\in H^0(X,E)$ such that $T$ is smooth of dimension $\dim X-\rank E$. Let $Y:=w^{-1}(0)\subset \PP(E)$ be the zero locus of the section $w\in H^0(\PP(E),\OO_{\PP(E)}(1))$ that corresponds to $s$ under the natural isomorphism $H^0(X,E)\cong H^0(\PP(E),\OO_{\PP(E)}(1))$. 
There are (correspondence-induced) isomorphisms of Chow groups
  \[  A_i(Y)\ \cong A_{i+1-r}(T) \oplus \bigoplus_{j=0}^{r-2} A_{i-j}(X) \ \ \ \forall\ i\ .\]
  
  In particular, there are isomorphisms
  \[  A^{hom}_i(Y)\ \cong\     A_{i+1-r}^{hom}(T)\oplus  \bigoplus_{j=0}^{r-2} A_{i-j}^{hom}(X) \ \ \ \forall\ i\ .\]
    \end{theorem}
    
    \begin{proof} The first statement is a special case of \cite[Theorem 3.1]{Ji} (the statement is actually true with {\em integer\/} coefficients). Both the isomorphism and its inverse are explicitly described.
    The crucial point is that the projection $Y\to X$ is a $\PP^{r-2}$-fibration over $X\setminus T$, and a $\PP^{r-1}$-fibration over $T$.
    
    As for the second statement, one observes that the first statement also holds on the level of Chow motives (this is \cite[Corollary 3.8]{Ji}). This implies that there is a commutative diagram (where vertical arrows are cycle class maps)
    \[ \begin{array}[c]{ccc}
         A_i(Y)& \cong&   A_{i+1-r}(T)\ \oplus\  {\displaystyle\bigoplus_{j=0}^{r-2}} A_{i-j}(X) \\
         &&\\
         \downarrow&&\downarrow\\
         &&\\
          H_{2i}(Y,\QQ)& \cong&\ \ \ \  H_{2i+2-2r}(T,\QQ)\ \oplus\  {\displaystyle\bigoplus_{j=0}^{r-2}} H_{2i-2j}(X,\QQ) \ .\\
          \end{array}\]   
   This proves the second statement.
     \end{proof}

\begin{remark} In the set-up of Theorem \ref{ji}, a cohomological relation between $Y$, $X$ and $T$ was established in \cite[Prop. 4.3]{Ko} (cf. also \cite[section 3.7]{IM0}, as well as \cite[Proposition 46]{BFM} for a generalization). A relation on the level of derived categories was established in \cite[Theorem 2.10]{Or} (cf. also \cite[Theorem 2.4]{KKLL} and \cite[Proposition 47]{BFM}).
\end{remark}

 \subsection{A variant of the Cayley trick}
 
 \begin{proposition}\label{variant} Let 
   \[ \begin{array}  [c]{ccc}   Y_T & \hookrightarrow & Y\\
                    &&\\
                  \ \   \downarrow{}&&\ \  \downarrow{\scriptstyle p}\\
                  &&\\
                  T & \hookrightarrow & X\\
                  \end{array}\]
                 be a Cartesian diagram of projective varieties, with $T\subset X$ of codimension $c$. Assume that $p$ is a proper morphism which is a $\PP^n$-bundle over $X\setminus T$, and a $\PP^m$-bundle over $T$. Assume also that there exists $h\in A^1(Y)$ such that $h\vert_{Y\setminus Y_T}$ is relatively ample for the $\PP^n$-bundle and $h\vert_{Y_T}$ is relatively ample for the $\PP^m$-bundle. Then there is 
                             an exact sequence
              \[   \bigoplus_{j=n+1}^{m}  A_{i -j}(T)\ \to\ A_i(Y)\ \to\  \bigoplus_{j=0}^n  A_{i-j}(X)\ \to\ 0\ .\]              
                \end{proposition}   
  
  \begin{proof} We use Bloch's higher Chow groups $A_i(-,j)$ \cite{B2}, \cite{B3}.
  There is a commutative diagram with long exact rows
      \[ \begin{array}[c]{ccccccc}
                       A_{i}(Y\setminus Y_T,1)& \xrightarrow{}&\ A_{i}(Y_T)& \xrightarrow{}& A_{i}(Y)& \to& A_i(Y\setminus Y_T) \ \ \ \to\ \ \ 0   \\
    &&&&&&\\
    \downarrow {\cong}&&\ \ \ \ \ \ \ \ \ \  \downarrow{\scriptstyle \Phi_T} &&\downarrow{\scriptstyle \Phi} &&\downarrow{\cong}\\
    &&&&&&\\
          {\displaystyle\bigoplus_{j=0}^n} A_{i-j}(X\setminus T,1) &\to &   {\displaystyle  \bigoplus_{j=0}^n} A_{i-j}(T)     &\xrightarrow{} &  {\displaystyle\bigoplus_{j=0}^n} A_{i-j}(X)& \to&  {\displaystyle\bigoplus_{j=0}^{n}} A_{i-j}(X\setminus T) \ \to\ 0\ .  \\
       \end{array}\]     
  
  The vertical arrows in this diagram are defined as follows: the map $\Phi$ is $\sum_{j=0}^n p_\ast ( h^j\cap -)$, and the map $\Phi_T$ is 
  $\sum_{j=0}^n p_\ast ( (h\vert_{Y_T})^j\cap -)$. Similarly, the left and right vertical maps are defined by restricting $h$ to $Y\setminus Y_T$. Commutativity of the diagram is proven as in \cite[diagram (10)]{KR}, by looking at the level of the underlying complexes. The left and right vertical arrows are isomorphisms because of the projective bundle formula for higher Chow groups \cite{B2}.
  
  The projective bundle formula says that $\Phi_T$ is surjective with kernel
    \[ \ker \Phi_T\ \cong\ {\displaystyle\bigoplus_{j=n+1}^{m}}  A_{i-j}(T)\ .\]
    A diagram chase now yields the desired exact sequence.
    \end{proof}

   \begin{remark} The case $(m,n)=(r-2,r-1)$ of Proposition \ref{variant} gives back a weak version of Jiang's result (Theorem \ref{ji}). Versions of Proposition \ref{variant} on the level of cohomology and on the level of derived categories are given in \cite[Appendix A]{BFM}, resp. \cite[Appendix B]{BFM}.
   
   At least when all varieties are {\em smooth\/}, it seems likely that a stronger version of Proposition \ref{variant} is true: we guess that in this case there is an isomorphism of Chow groups
   \[  A_i(Y)\ \cong\bigoplus_{j=n+1}^{m}  A_{i -j}(T)\ \oplus\ \bigoplus_{j=0}^n  A_{i-j}(X)\ .\]
   Since this is not needed below, we have not pursued this guess.
    \end{remark}

   \subsection{Spreading out}
   
   \begin{proposition}\label{spread} Let $\YY\to B$ be a family of smooth projective varieties. Assume there is some $c\in\NN$ that
      \begin{equation}\label{vanis} A_i^{hom}(Y_b)=0\ \ \ \forall i\le c\ \end{equation}
      for the very general fiber $Y_b$. Then
     \[ A_i^{hom}(Y_b)=0\ \ \ \forall i\le c\ \]
     for every fiber $Y_b$.
       \end{proposition}
   
   \begin{proof} Let $B^\circ\subset B$ denote the intersection of countably many Zariski open subsets such that the vanishing
   \eqref{vanis} holds for all $b\in B^\circ$.
   
   Doing the Bloch--Srinivas argument \cite{BS} (cf. also \cite{moi}), this implies that for each $b\in B^\circ$ one has
   a decomposition of the diagonal
     \begin{equation}\label{decom} \Delta_{Y_b} = \gamma_b+\delta_b \ \ \ \hbox{in}\ A^{\dim Y_b}(Y_b\times Y_b)\ \end{equation}
     where $\gamma_b$ is completely decomposed (i.e. $\gamma_b\in A^\ast(Y_b)\otimes A^\ast(Y_b)$) and $\delta_b$ is supported on
     $Y_b\times W_b$ with $\codim \, W_b=c+1$.
        
     Using Hilbert schemes as in the proofs of \cite[Theorem 2.1(\rom1)]{V10}, \cite[Proposition 3.7]{V0}, the fiberwise data 
     \[  (\gamma_b, \, \delta_b, \, W_b) \]
     can be encoded by a countably infinite number of varieties, each carrying a universal object. By a Baire category argument, one of these varieties must dominate $B$.
     Taking a linear section, this means that after a generically finite base change the 
     $\gamma_b, \delta_b, W_b$ exist relatively. That is, there exist a generically finite morphism $B^\prime\to B$, 
   a cycle $\gamma\in A^\ast(\YY^\prime\times_{B^\prime}\YY^\prime)$ that is completely decomposed (i.e. $\gamma\in (p_1)^\ast A^\ast(\YY^\prime)\cdot (p_2)^\ast A^\ast(\YY^\prime)$), a subvariety $ \WW\subset \YY^\prime$ of codimension $c+1$, and a cycle $\delta$ supported on $\YY^\prime\times_{B^\prime} \WW$ such that
        \[ \Delta_{\YY^\prime}\vert_b= \gamma\vert_b + \delta\vert_b\ \ \ \hbox{in}\ A^{\dim Y_b}(Y_b\times Y_b)\ \ \ \forall\ b\in B^\circ\ .\ \]
        (Here $\YY^\prime:=\YY\times_B B^\prime$.)
        
     Let $\bar{\gamma}, \bar{\delta} \in A^{\dim Y_b}(\YY^\prime\times_{B^\prime} \YY^\prime)$ be cycles that restrict to $\gamma$ resp. $\delta$. The spread lemma \cite[Proposition 2.4]{V10}, \cite[Lemma 3.2]{Vo} then implies that
      \[ \Delta_{\YY^\prime}\vert_b= \bar{\gamma}\vert_b + \bar{\delta}\vert_b\ \ \ \hbox{in}\ A^{\dim Y_b}(Y_b\times Y_b)\ \ \ \forall\ b\in B\ .\ \]     
   Given any $b_1\in B\setminus B^\circ$, one can find representatives for $\bar{\gamma}$ and $\bar{\delta}$ in general position with respect to the fiber $Y_{b_1}\times Y_{b_1}$.   
Restricting to the fiber, this implies that the diagonal of $Y_{b_1}$ has a decomposition as in \eqref{decom}, and so \eqref{decom} holds for all $b\in B$.
 Letting the decomposition \eqref{decom} act on Chow groups, this shows that
     \[  A_i^{hom}(Y_b)=0\ \ \ \forall\ i\le c\ ,\ \ \ \forall\ b\in B\ .\]
    \end{proof}

   \section{Main results}
   
   \begin{theorem}\label{main} Let 
  \[ Y= \Gr(k,n)\cap H \ \ \ \subset\ \PP^{{n\choose k}-1} \]
  be a smooth hyperplane section (with respect to the Pl\"ucker embedding). Then
    \[ A_i^{hom}(Y)=0\ \ \ \forall i\le n-k\ .\]
     \end{theorem}
   
   \begin{proof} In view of Proposition \ref{spread}, it suffices to prove this for generic hyperplane sections, and so we may assume $Y$ is as in Proposition \ref{jump}.
   The {\em jump\/} of Proposition \ref{jump} gives rise to a commutative diagram
     \[ \begin{array}[c]{ccccc}
           F & \hookrightarrow& q^\ast Y & \xrightarrow{q} &    Y\\
           &&&&\\
          \downarrow{} &&\ \ \downarrow{\scriptstyle p}&&\\
           &&&&\\
           T & \hookrightarrow & \ \Gr(k-1,n)\ .&&\\
           \end{array} \]
  The morphism $q^\ast Y\to Y$ is a $\PP^{k-1}$-bundle, and so the projective bundle formula implies there are injections
   \begin{equation}\label{inj}  A_i^{hom}(Y)\ \hookrightarrow\ A_i^{hom} (q^\ast Y)\ \ \ \forall i\ .\end{equation}
  
 For $Y$ sufficiently general,  the locus $T$ will be smooth of codimension $n-k+1$ (Proposition \ref{jump}). The set-up is thus that of Cayley's trick, with $X=\Gr(k-1,n)$ and $E=Q^\ast(1)$. Applying
     Theorem \ref{ji} (with $r=\rank Q^\ast(1)=n-k+1$), we find that
    there are isomorphisms
      \begin{equation}\label{iso}  A_i^{hom}(q^\ast Y)\ \cong\  A_{i-n+k}^{hom}(T)\ .\end{equation}
      
  But $T$ is a smooth Fano variety (indeed, using adjunction one can compute that the canonical bundle of $T$ is $\OO_T(1-k)$), hence $T$ is rationally connected \cite{Cam}, \cite{KMM} and so in particular $A_0(T)\cong\QQ$. It follows that
    \begin{equation}\label{van}     A_{i-n+k}^{hom}(T)=0\ \ \ \forall\ i\le n-k\ .\end{equation}
    Combining \eqref{inj}, \eqref{iso} and \eqref{van}, the theorem is proven.
%
        \end{proof}

   \begin{theorem}\label{main2} Let 
  \[ Y= \Gr(3,n)\cap H \ \ \ \subset\ \PP^{{n\choose 3}-1} \]
  be a smooth hyperplane section (with respect to the Pl\"ucker embedding).  Assume $n\le 13$, $n\not=12$.
  Then
    \[ A_i^{hom}(Y)=0\ \ \ \forall i\le n-2\ .\]
   
   \end{theorem} 
    
   \begin{proof} In view of Theorem \ref{main}, it only remains to treat the case $i=n-2$.
   
  Applying Proposition \ref{jump}, we may assume $Y$ is sufficiently generic. Doing the jump as in the proof of Theorem \ref{main} above, one finds a smooth variety $T\subset\Gr(2,n)$ (of codimension $n-2$) and an injection of Chow groups
     \begin{equation}\label{inj2}  A_{n-2}^{hom}(Y)\ \hookrightarrow\ A_{1}^{hom}(T)\ .\end{equation}
     
 In order to understand $A_{1}^{hom}(T)$, we perform a second jump. This second jump (cf. \cite[Section 3.4]{BFM}) induces a diagram
       \[ \begin{array}[c]{ccccc}
           F & \hookrightarrow& q^\ast T & \xrightarrow{q} &    T\\
           &&&&\\
         \ \ \  \ \ \  \downarrow{\scriptstyle p\vert_F} &&\ \ \downarrow{\scriptstyle p}&&\\
           &&&&\\
           P^\prime & \hookrightarrow &  \ P\ .&&\\
           \end{array} \]
    Here $q^\ast T\subset Fl(1,2,n)$, and the morphism $q^\ast T\to T$ is a $\PP^1$-bundle. The projective bundle formula gives an injection
      \begin{equation}\label{inj3}  A_{1}^{hom}(T)\ \hookrightarrow\ A_{1}^{hom}(q^\ast T)\ .\end{equation}
   The varieties $P$ and $P^\prime$ depend on the parity of $n$:
  
  \begin{itemize}
  
  \item If $n$ is even, $P=\PP^{n-1}$ and $P^\prime$ is the $(n-4)$-dimensional Pfaffian variety called $P(1,n)$ in \cite{BFM}. For $n\le 10$, the generic $P(1,n)$ is smooth and in this case $p$ is the blow-up of $P=\PP^{n-1}$ with center $P^\prime=P(1,n)$.
  
  \item If $n$ is odd, $P$ is the Pfaffian hypersurface $P(1,n)\subset \PP^{n-1}$ (in the notation of \cite{BFM}). For $n\le 15$, the generic $P(1,n)$ has singular locus $P^\prime\subset P(1,n)$ of codimension $5$ and $P^\prime$ is smooth. In this case, the morphism $p$ is a $\PP^1$-bundle over $ P\setminus P^\prime$, and a $\PP^3$-bundle over $P^\prime$.
  \end{itemize}
  
 Let us first treat the case $n$ even, $n\le 10$. The blow-up formula gives an isomorphism
   \[ A_1^{hom}(q^\ast T)= A_0^{hom}(P^\prime)\ .\]
   Since $P^\prime=P(1,n)$ is a smooth Fano variety, we have
   \[A_0^{hom}(P^\prime)=0\ ,\] 
 and so (combining with \eqref{inj2} and \eqref{inj3}) the theorem follows for $n$ even and generic $Y$.
   
   Next, let us treat the case $n$ odd, $n\le 13$. In this case, we apply Proposition \ref{variant} (with $h\in A^1(q^\ast T)$ the restriction of the relatively ample class for $Fl(1,2,n)\to \PP^{n-1}$). This
   gives a (correspondence-induced) isomorphism
    \[ A_1^{}(q^\ast T) = A^{}_0(P)   \oplus     A^{}_1(P) \ ,\]   
    and in particular an injection
     \[ A_1^{hom}(q^\ast T) \ \hookrightarrow\ A^{hom}_0(P)   \oplus     A^{hom}_1(P) \ .\]
     But $P=P(1,n)\subset\PP^{n-1}$ is a (singular) hypersurface of degree $(n-3)/2$. For $n\le 13$, it is known that any hypersurface $P\subset\PP^{n-1}$ of degree $\le (n-3)/2$ has
     \[  A_0^{hom}(P)=  A_1^{hom}(P)=0\ .\] 
(For smooth $P$ this is proven in \cite{Ot}, the extension to singular $P$ is done in \cite{HI}. Note that, at least for smooth $P$, Conjecture \ref{conj0} states that the restriction to $n\le 13$ is not necessary.) Combined with \eqref{inj2} and \eqref{inj3}, the theorem follows for $n$ odd and $Y$ generic.    
    \end{proof}

   \section{Some consequences}
   
   \begin{corollary}\label{ghc} Let $Y$ be as in Theorem \ref{main2}. Then $H^{\dim Y}(Y,\QQ)$ is supported on a subvariety of codimension $n-1$.
   \end{corollary} 
   
   \begin{proof} This follows in standard fashion from the Bloch--Srinivas argument. The vanishing of Theorem \ref{main2} is equivalent to the decomposition
     \[  \Delta_Y = \gamma+\delta\ \ \ \hbox{in}\ A^{\dim Y}(Y\times Y)\ ,\]
     where $\gamma$ is a completely decomposed cycle (i.e. $\gamma\in A^\ast(Y)\otimes A^\ast(Y)$), and $\delta$ has support on $Y\times W$ with $W\subset Y$ of codimension $n-1$.
     Let $H^{\dim Y}_{tr}(Y,\QQ)$ denote the transcendental cohomology (i.e. the complement of the algebraic part under the cup product pairing). The cycle $\gamma$ does not act on
     $H^{\dim Y}_{tr}(Y,\QQ)$. The action of $\delta$ on $H^{\dim Y}_{tr}(Y,\QQ)$ factors over $W$, and so
     \[  H^{\dim Y}_{tr}(Y,\QQ)\ \ \subset\ H^{\dim Y}_{W}(Y,\QQ)\ .\]
     Since the algebraic part of $H^{\dim Y}(Y,\QQ)$ is (by definition) supported in codimension $\dim Y/2$, this settles the corollary.
             \end{proof}
   
    \begin{corollary}\label{cor2} Let  
\[ Y= \Gr(3,9)\cap H \ \ \ \subset\ \PP^{83} \]
  be a smooth hyperplane section (with respect to the Pl\"ucker embedding). 
  Then 
    \[ A^\ast_{AJ}(Y)=0\ .\]
  In particular,
  $Y$ has finite-dimensional motive (in the sense of \cite{Kim}, \cite{An}).    
  \end{corollary}
  
  \begin{proof} Theorem \ref{main2} implies that
    \[ A_i^{hom}(Y)=0\ \ \ \forall\ i\le 7\ .\]
    The Bloch--Srinivas argument \cite{BS}, \cite{moi} then implies that
    \[ A^j_{AJ}(Y)=0\ \ \ \forall\ j\le 9\ .\]
    Since $Y$ is $17$-dimensional, these two facts taken together mean that 
      \[ A^\ast_{AJ}(Y)=0\ ,\]
      as claimed.
      The fact that any variety $Y$ with $A^\ast_{AJ}(Y)=0$ is Kimura finite-dimensional is \cite[Theorem 4]{43}.
      \end{proof}

 \vskip0.5cm
\begin{nonumberingt} Thanks to Kai and Len for making me listen to S\'ebastien Patoche. Thanks to the referee for many constructive comments that significantly improved the paper.
\end{nonumberingt}

\end{document}